\documentclass[onenumber,plainnumbering,english,11 pt]{amsart}
\usepackage[utf8]{inputenc}
\usepackage[square,numbers]{natbib}
\usepackage[mathscr]{eucal}
\usepackage{amssymb,amsmath,amscd,mathrsfs,paralist,amsthm}
\usepackage{mathrsfs}
\usepackage{bm}
\usepackage{scalerel}
\usepackage{frcursive}
\usepackage[T1]{fontenc}
\usepackage{mathdots}% http://ctan.org/pkg/mathdots
\usepackage{MnSymbol}
\usepackage{tikz}
\usepackage{listings}
\usetikzlibrary{shapes.geometric, arrows}                      
\usepackage{tikz}
\usepackage{scalerel,stackengine}
\usetikzlibrary{matrix}
\usetikzlibrary{cd}
\usepackage[colorlinks,linkcolor=black,urlcolor=black]{hyperref}
\hypersetup{
     colorlinks   = true,
     citecolor    = black}
\usepackage{graphicx}
\usepackage{bbm}
\usepackage[all,cmtip]{xy}
\theoremstyle{plain}
\newtheorem{theorem}{Theorem}[section]

\newtheorem{lemma}[theorem]{Lemma}

\theoremstyle{definition}
\newtheorem{definition}[theorem]{Definition}

\newtheorem{example}[theorem]{Example}

\newtheorem{remark}[theorem]{Remark}

\usepackage{bm}

\newcommand{\D}{{\mathscr H}}

\numberwithin{equation}{section}
\newcommand{\N}{\mathbb N}

\newcommand{\C}{\mathbb C}

\renewcommand{\k}{{\mathbf k}}

\newcommand{\theoref}[1]{Theorem~\ref{#1}}

\newcommand{\lemref}[1]{Lemma~\ref{#1}}

\newcommand{\exref}[1]{Example~\ref{#1}}

\newcommand{\secref}[1]{Section~\ref{#1}}
\usepackage[all]{xy}
\usepackage{lipsum}  
\makeatletter
\@namedef{subjclassname@2020}{\textup{2020} Mathematics Subject Classification}
\makeatother
\usepackage{microtype}

  \newcommand{\la }{\langle}
\newcommand{\ra }{\rangle}

\newcommand{ \cf }{\emph{cf.} }

\begin{document}
\title{Cross-ratios and first order ordinary differential equations}
\keywords{Generalized Riccati-Abel equations, Cross-ratio, Schur functions, Algebraic solutions}
\subjclass[2020]{12H05, 34M15}
\author{Oumar Wone}
\address{Oumar Wone\\Département de Mathématiques Université Cheikh Anta Diop (UCAD), Avenue Cheikh Anta Diop, Dakar Sénégal.}
    \email{omzowone@gmail.com}   
 \begin{abstract}
We investigate the usefulness of the cross-ratio in the study of first order odes of the form $y^\prime=P_ny^n+P_{n-1}y^{n-1}+\ldots+P_1y+P_0,\;\; (P_i)_{0\leqslant i\leqslant n}\in K,\;n\geqslant1$, where $K$ is a differential field of characteristic zero with algebraically closed field of constants. We discover along the way many differential algebraic identity between solutions to that equation. We also explain why the cross-ratio of any four solutions to such an equation is in general not constant when $n\geq3$ by exhibiting an obstruction which prevents its constancy when $n\geq3$. Finally to any separable polynomial $P_\varphi(z)=z^{n+1}+\varphi_1z^n+\varphi_2z^{n-1}+\ldots+\varphi_{n+1}$, $n\geq2$, over a differential field of meromorphic function on some domain of $\C$ with the usual derivation we associate a uniquely defined first order ode of the form $y^\prime=P_ny^n+P_{n-1}y^{n-1}+\ldots+P_1y+P_0,\;\; (P_i)_{0\leqslant i\leqslant n}\in K$ such that every root of $P_\varphi(z)$ is a solution to the said differential equation.
\end{abstract} 
\maketitle
\tableofcontents
\section{Introduction}
We learn in the old treatises, cf. for instance \cite{hille1976, ince1944}, that the cross-ratio of any four distinct solutions to a Riccati equation is a constant. This property is related to the theory of Vessiot-Guldberg-Lie algebras (Lie superposition) or what is the same in modern parlance to the question of internality to the constants of first order odes $y^\prime=f(y)$ where $f(y)\in K[y]$ and $(K,\delta)$ is a differential field of characteristic zero with algebraically closed field of constants $C_K=\ker \delta$, say. In fact the Riccati equation $y^\prime=P_2y^2+P_1y+P_0$ and its sister equations $y^\prime=Q_1y+Q_0$, $y^\prime=R_0$ are the only first order ODEs with polynomial right hand sides which admit a Lie superposition formula aka are internal to the constants, see \cite[§11.2]{ibragimov1999} for the classical proof or \cite{srinivasan2026, nagloo} for a modern algebraic proof. Above the $P_i$, $Q_j$, $R_k$ belong to $K$. Since the Riccati equation is the only first order $y^\prime=f(y)$ where $f(y)\in K[y]$ having such a property it is hopeless to search for a similar property for first order odes
$$y^\prime=P_ny^n+P_{n-1}y^{n-1}+\ldots+P_1y+P_0,\quad (P_i)_{0\leqslant i\leqslant n}\in K,\;n\geq3,\;P_n\not=0.$$
But rather than internality to the constants is it possible to exhibit differential algebraic relations between solutions to that last differential equation in some differential field extension which could possibly play the role of internality. One hint in that direction is given by considering the logarithmic derivative of the cross-ratio of any four distinct solutions to a Riccati differential equation with $P_2\not=0$. Then if we denote still by $\delta$ the extension of the derivation on $K$ to a derivation on $K(y_1,y_2,y_3,y_4)$ obtained by setting $y_i^\prime=P_2y_i^2+P_1y_i+P_0$ for $i=1,2,3,4$, we see that the constancy of solutions to a given Riccati equation is equivalent to the vanishing of the logarithmic derivative with respect to $\delta$ given by
$$d\log\left(\dfrac{(y_3-y_1)(y_4-y_2)}{(y_3-y_2)(y_4-y_1)}\right)=0.$$
It is through this route, i.e. by considering logarithmic derivatives of differences of solutions and their differences that we find differential algebraic relations for differential equations
$$y^\prime=P_ny^n+P_{n-1}y^{n-1}+\ldots+P_1y+P_0,\quad (P_i)_{0\leqslant i\leqslant n}\in K,\;n\geq3,\;P_n\not=0$$
which could play the role of the constancy of the cross-ratio for the Riccati equation, \cf. \theoref{curious}. One the other hand what exactly is preventing a generic given differential equation 
$$y^\prime=P_ny^n+P_{n-1}y^{n-1}+\ldots+P_1y+P_0,\quad (P_i)_{0\leqslant i\leqslant n}\in K,\;n\geq3,\;P_n\not=0$$
from having a constant cross-ratio. We find in \theoref{obstruction} the exact obstruction which prevents such an occurrence. This obstruction is expressed as a $K$-linear combination of certain determinants on the solutions which arise via the use of Schur functions. For the theory of Schur functions we refer to \cite[p. 40]{macdonald1995}.

Besides classically since the time of Schwarz \cite{schwarz1872} and Liouville \cite{liouville1835}, given a differential equation over a differential field $(K,\delta)$ one seeks whether or not it admits solutions of a given type, for instance one looks for solutions algebraic over $K$ or Liouvillian over $K$ etc, see also \cite{katz1972} and \cite{singer1992} for modern undertakings of such problem. But what about inverting that type of problems? For instance for a given separable polynomial $P_a=z^{n+1}+a_1(x)z^n+a_2(x)z^{n-1}+\ldots+a_{n+1}(x)$ of degree $n+1$, $n\geq2$ over a differential field $(K,\frac{d}{dx})$ of meromorphic functions on a domain, and such that $\ker\frac{d}{dx}=\C$, is it possible to find a differential equation of the 
$$y^\prime=P_ny^n+P_{n-1}y^{n-1}+\ldots+P_1y+P_0,\quad (P_i)_{0\leqslant i\leqslant n}\in K$$
such that all the roots of the polynomial $P_a$ are at the same time solutions to the differential equation? Formulated as such the problem looks like some kind of Riemann-Hilbert problem. We show that given such a separable polynomial $P$ there is exactly one unique such differential equation
$$y^\prime=P_ny^n+P_{n-1}y^{n-1}+\ldots+P_1y+P_0,\quad (P_i)_{0\leqslant i\leqslant n}\in K.$$
We refer to \theoref{brm}. Our approach to this question is akin but distinct to the one taken in the theory of GKZ hypergeometric functions \cite{sturmfels2000}, i.e. we consider the roots of $P_a$ as being functions depending on the coefficients of $P_a$ and on $x$, see \cite[§2]{sturmfels2000} for a similar reasoning.

The plan of this article is as follows in \secref{differential} after some preliminaries to fix the notations and terminology we explain how to get \theoref{curious} starting from first principles. Then in \secref{obstructionto} after indicating and justifying how Schur functions arise in our context we give the proof of \theoref{obstruction}. Finally in \secref{algebraic} we first of all explain how to get our key identity \eqref{pol27} in \lemref{lemrefAE} and after that we deduce fairly quickly the \theoref{brm}. Following that we give at the very end two examples \exref{onetoone} and \exref{onetotwo} motivating the theory presented in the \secref{algebraic}.
\section{Some differential algebraic relations over the field of constants between solutions of a generalized Riccati-Abel equation}
\label{differential}
We begin with some preliminaries which shall be needed in many parts of this article. A differential field $K$ is a field $K$ endowed with a derivation that is an abelian group homomorphism $\delta:K\to K$ which also satisfies the Leibniz rule. A finitely generated differential field extension $F$ of $K$ is a differential field obtained from $K$ by considering the quotient field of the ring formed from $K$ by the adjonction of a finite number of differential indeterminates, meaning that there are some elements $w_1,w_2,\ldots,w_m$ of $F$ such that $F$ is the polynomial ring over $K$ with an infinite number of variables given by the $w_i$ and all their derivatives, see \cite{ritt1950, kaplansky1957}. Such a differential field extension is in general denoted $K\la w_1,w_2,\ldots,w_m\ra$. The underlying ring of a finitely generated differential field is called a finitely generated differential ring. An element of a finitely generated differential ring is called a differential polynomial.  Among the finitely generated differential rings over $K$, in case $m=1$ we have what is called the ring of ordinary differential polynomials denoted $K\{y\}$. Let $(K,\delta=:^\prime)$ be a differential field of characteristic zero with algebraically closed field of constants $\ker\delta$. A universal differential field extension $\mathscr U$ \cite[§5, p. 769]{kolchin1953} of $K$ is differential field extension such that for any finitely generated differential field extension $F\subseteq\mathscr U$ of $K$ and given any finitely generated differential field extension $F_1$ of $F$, which is not necessarily contained in $\mathscr U$, there is a differential $F$-embedding of $F_1$ into $\mathscr U$, that is a differential embedding $\sigma$ (commuting with derivations) with $\sigma(a)=a$ for all $a\in F$. Every differential field $K$ of characteristic zero has a universal extension $\mathscr U$ by \cite[§5]{kolchin1953}. Furthermore $\mathscr U$ is algebraically closed, and the field of constants $C_{\mathscr U}$ of $\mathscr U$ is an algebraically closed extension of the field of constants $C_K$ of $K$ of infinite degree of transcendence, see \cite[§5]{kolchin1953}.

Let $K\{y\}$ be the ring of ordinary differential polynomials over $K$. Let $\Sigma$ be any system of differential polynomials in $K\{y\}$. A zero $\eta$ of $\Sigma$ is an element of differential field extension of $K$ which annuls every differential polynomial in $\Sigma$, \cite[p. 21]{ritt1950}. By definition of $\mathscr U$ such an $\eta$ can be considered as belonging to $\mathscr U$ since there is a differential embedding over $K$ of $K\la\eta\ra$ into $\mathscr U$. More generally $\mathscr U$ "contains" every solution to any differential polynomial $P\in K\{y\}\setminus K$ in any differential field extension of $K$.

Let $P\in K\{y\}\setminus K$ be a differential polynomial. Then the order of $P$, denoted $\rm{ord}(P)$ is the highest derivative of $y$ effectively present in $P$. Then the separant $S_P$ of $P$ is by definition $S_P:=\frac{\partial P}{\partial y^{(\rm{ord}(P))}}$. A differential polynomial $P\in K\{y\}\setminus K$ is called irreducible when $P$ is irreducible as a polynomial in $K\left[y,y^{\prime}, y^{\prime\prime},\ldots,y^{(\rm{ord}(P))}\right]$. Let $P\in K\{y\}\setminus K$ be an irreducible differential polynomial. Then one defines \cite[p. 26, §6]{ritt1950} the differential ideal (stable under derivation)
$$\Sigma_P:=\{A\in K\{y\}|S_PA\equiv0\mod\{P\}\}$$
where $\{P\}$ is the radical (perfect) differential ideal generated by $P$, see \cite[p. 11]{kaplansky1957}. One shows, cf. \cite[lem. 7.9]{kaplansky1957}, \cite{aroca}, \cite[p. 30]{ritt1950}, that $\Sigma_P$ is a prime differential ideal if $P$ is irreducible. Let now $\Sigma$ be any prime differential ideal of $K\{y\}$ distinct from the unit ideal. Then a zero $\eta$ of $\Sigma$ is called generic if for any differential polynomial $P$, $P(\eta)=0$ implies that $P\in\Sigma$, cf. \cite[p. 26-27, p. 129-131]{ritt1950}, \cite[prop.1, p. 146]{kolchin1973}. Furthermore every prime differential ideal $\Sigma$ of $K\{y\}$ which is a non-unit prime differential ideal has a generic zero by \cite[p. 27]{ritt1950}. Using this we may define, see \cite{ritt1950}, for a given irreducible differentiable polynomial $P\in K\{y\}\setminus K$ a general or generic solution $\eta$ of $P$ as a generic solution $\eta$ of $\Sigma_P$. Moreover by \cite[§9, p. 138-139]{kolchin1973} every other solution $\xi$ to $\Sigma_P$ arises from $\eta$ via differential specialization, i.e. from the consideration of a morphism of differential rings $f:K\{\eta\}\to K\la\xi\ra$ with $f(\eta)=\xi$.

Let $(K,\delta=:^\prime)$ be a differential field of characteristic zero with algebraically field of constant $\rm C_K=\ker(\delta)$. In this section we aim to generalize in various ways known properties of solutions to the Riccati equation to differential equations of the form
\begin{equation}
\label{pol0}
y^\prime=P_ny^n+P_{n-1}y^{n-1}+\ldots+P_1y+P_0,\quad (P_i)_{0\leqslant i\leqslant n}\in K,\;n\geqslant1.
\end{equation}
Before beginning the study we recall what a Picard-Vessiot extension of a differential field is since we shall need that along the way.
\begin{definition}
Let $(K,\delta)$ be a differential field of characteristic zero and algebraically closed field of constants $C_K$. A Picard-Vessiot extension of $K$ is a differential field extension with derivation still denoted $\delta$ such that 
\begin{enumerate}
\item $\delta:\mathscr K\to\mathscr K$ has kernel $\ker\delta=C_K$
\item $\mathscr K=K\langle\omega\rangle$ where $(\omega_1,\ldots,\omega_n)$ is a fundamental system of solutions of a linear homogeneous differential equation
\begin{equation}
\label{tolly13}
\mathscr L(y)=y^{(n)}+a_1y^{(n-1)}+\ldots+a_0y=0:\end{equation}
$\omega_1$, $\omega_2$, $\ldots$, $\omega_n$ are $n$ solutions to \eqref{tolly13} linearly independent over $C_K$. Furthermore by definition $K\langle\omega\rangle=K(\omega_1,\ldots,\omega_n,\ldots,\omega_1^{(n-1)},\ldots,\omega_n^{(n-1)})$. 
\end{enumerate}
See \cite[p. 21-22]{kaplansky1957}, \cite[§1.3, p. 12-13]{singer2003}.
\end{definition}

\subsection{Case $n=1$}
Given the differential equation
\begin{equation}
\label{pol00}
y^\prime=P_1y+P_{0},\quad (P_i)_{0\leqslant i\leqslant 1}\in K.
\end{equation}
Let $y$ be the general solution and $y_j$ be a particular solution to \eqref{pol00} in a differential field extension of $K$. Then
\begin{equation}
\label{pol01}
y^\prime-y^\prime_i=P_1(y-y_j).
\end{equation}
So if $y_1$ and $y_2$ are any distinct particular solutions distinct to \eqref{pol00} we obtain 
\begin{equation}
\label{pol02}
\begin{split}
y^\prime-y^\prime_1=P_1(y-y_1).\\
y^\prime-y^\prime_2=P_1(y-y_2).
\end{split}
\end{equation}
Hence by taking the difference we arrive at
\begin{equation}
\label{pol03}
d\log(y-y_1)-d\log(y-y_2)=0
\end{equation}
where 
$$d\log(y-y_i):=\dfrac{\delta(y-y_i)}{y-y_i},\quad i=1,\,2.$$
Equation \eqref{pol3} when specialized to the tuple $(y_1,y_2,y_3)$ of three distinct solutions to the equation \eqref{pol00} in place of $(y,y_1,y_3)$
gives
\begin{equation}
\label{jnkknknjgb}
d\log\left(\dfrac{y_1-y_2}{y_1-y_3}\right)=0.
\end{equation}
Assume $P_0\not=0$ then one may readily show that any three distinct solutions of \eqref{pol00} are also solutions to the second-order linear homogeneous differential equation given by, see \cite[exer. 1.14, p. 10]{singer2003},
\begin{equation}
\label{tolly}
u^{\prime\prime}-\left(P_1+\dfrac{\delta(P_0)}{P_0}\right)u^\prime+\left(P_1\left(\dfrac{\delta P_0}{P_0}\right)-\delta(P_1)\right)u=0.\end{equation}
Thus we may achieve
$$\dfrac{y_1-y_2}{y_1-y_3}\in C_K=\ker\delta$$
if choose them to lie in the Picard-Vessiot extension of that second-order linear homogeneous differential equation: any solution to \eqref{pol00} is also a solution to \eqref{tolly}.

 As a byproduct of that we see that in that case the transcendence degree over $K$ of the field $K(y_1,y_2,y_3)$ is at most $2$ that is: $\rm{tr}.\deg (K(y_1,y_2,y_3)|K)\leqslant2$.
\subsection{Case $n=2$: Riccati equations}
Let us now start our study of the next case after the case of \eqref{pol00}. This is the case of the Riccati equation
\begin{equation}
\label{pol04}
y^\prime=P_2y^2+P_{1}y+P_0,\quad (P_i)_{0\leqslant i\leqslant 2}\in K,\; P_2\not=0.
\end{equation}
Let $y$ be the general solution and $y_j$ a particular solution to \eqref{pol04} in a differential field extension of $K$. Then 
\begin{equation}
\label{pol05}
y^\prime-y_j^\prime=P_2(y^2-y_j^2)+P_1(y-y_j).
\end{equation}
Thus for $2$ distinct particular solutions $y_1$ and $y_2$ we deduce from \eqref{pol05}
\begin{equation}
\label{pol06}
d\log(y-y_1)-d\log(y-y_2)=P_2(y_1-y_2).
\end{equation}
One may deduce quickly from \eqref{pol06} the constancy of the cross-ratio of any $4$ distinct solutions to \eqref{pol04}. In fact let $y_3$, $y_4$ be two other distinct solutions to \eqref{pol04} in a differential field extension, distinct from $y_1$ and $y_2$. Then replacing successively $y$ by $y_3$, $y_4$ we find that
\begin{equation}
\begin{split}
\label{pol07}
d\log(y_3-y_1)-d\log(y_3-y_2)=P_2(y_1-y_2)\\
d\log(y_4-y_1)-d\log(y_4-y_2)=P_2(y_1-y_2).
\end{split}
\end{equation}
Taking the difference of the $2$ equations in \eqref{pol07} we find
\begin{equation}
\label{pol08}
d\log\left(\dfrac{(y_3-y_1)(y_4-y_2)}{(y_3-y_2)(y_4-y_1)}\right)=0,
\end{equation}
i.e. the constancy of the cross-ratio: $\frac{(y_3-y_1)(y_4-y_2)}{(y_3-y_2)(y_4-y_1)}\in\ker\delta$. Here $\delta$ represents now the unique extension of the original derivation $\delta$ to $K(y_1,y_2,y_3,y_4)$. But actually we may choose $y_1$, $y_2$, $y_3$, $y_4$ in a same Picard-Vessiot extension, the one associated to the order two linear homogeneous ode
\begin{equation}
\label{khadija}
u^{\prime\prime}-\left(P_1+\dfrac{\delta(P_2)}{P_2}\right)u^\prime+P_0P_2u=0.
\end{equation}
This is because there is a surjection between the set of non-zero solutions to the associated second order ode \eqref{khadija}, and the solutions to the Riccati equation given by $u\mapsto\frac{-u^\prime}{P_2u}$. Thus we may achieve $\frac{(y_3-y_1)(y_4-y_2)}{(y_3-y_2)(y_4-y_1)}\in\ker\delta=\mathrm C_K$. Hence in that case as an upshot we find in particular that the transcendence degree of the field $K(y_1,y_2,y_3,y_4)$ over $K$ is at most $3$: $\rm{tr}.\deg (K(y_1,y_2,y_3,y_4)|K)\leqslant3$.

Now from \eqref{pol06} with $y_1$, $y_2$ replaced by $y_2$, $y_3$ such that the set the three solutions $y_1$, $y_2$, $y_3$ belonging to a differential field extension of $K$ are all distinct, we derive similarly
\begin{equation}
\label{pol09}
d\log(y-y_2)-d\log(y-y_3)=P_2(y_2-y_3).
\end{equation}
Therefore we discover by taking the difference between equations \eqref{pol06} and \eqref{pol09}
\begin{equation}
\label{pol010}
\dfrac{1}{(y_1-y_2)}d\log\left(\dfrac{y-y_1}{y-y_2}\right)-\dfrac{1}{(y_2-y_3)}d\log\left(\dfrac{y-y_2}{y-y_3}\right)=0
\end{equation}
which gives after multiplication by $(y_1-y_2)(y_2-y_3)$
\begin{equation}
\label{pol011}
(y_2-y_3)d\log\left(\dfrac{y-y_1}{y-y_2}\right)-(y_1-y_2)d\log\left(\dfrac{y-y_2}{y-y_3}\right)=0,
\end{equation}
that is
\begin{equation}
\label{pol012}
(y_2-y_3)d\log(y-y_1)-(y_1-y_3)d\log(y-y_2)+(y_1-y_2)d\log(y-y_3)=0.
\end{equation}
This last formula generalizes \eqref{pol03}. If we eliminate the $(P_i)_{0\leqslant i\leqslant2}$ between the equations
\begin{equation}
\label{pol0012}
y_i^\prime=P_2y_i^2+P_1y_i+P_0
\end{equation} 
for distinct $(y_i)_{1\leqslant i\leqslant4}$, then we see that
\begin{equation}
\label{pol0013}
\left|\begin{array}{cccc}y_1^\prime & y_1^2 & y_1 & 1 \\y_2^\prime & y_2^2 & y_2 & 1 \\y_3^\prime & y_3^2 & y_3 & 1 \\y_4^\prime & y_4^2 & y_4 & 1\end{array}\right|=0,
\end{equation}
where $|\cdot|$ means the determinant of what is enclosed inside. The equation \eqref{pol0013} gives after computation
\begin{equation}
\label{pol0014}
(y_1-y_2)(y_3-y_4)\delta((y_1-y_4)(y_3-y_2))-(y_1-y_4)(y_3-y_2)\delta((y_1-y_2)(y_3-y_4))=0.
\end{equation}
That is 
\begin{equation}
\label{pol0015}
((y_1-y_4)(y_3-y_2))^2\delta\left(\dfrac{(y_1-y_2)(y_3-y_4)}{(y_1-y_4)(y_3-y_2)}\right)=0.
\end{equation}
Hence we discover again the constancy of the cross-ratio of four distinct solutions to \eqref{pol04}. Again as reminded before one may arrange so that $\frac{(y_3-y_1)(y_4-y_2)}{(y_3-y_2)(y_4-y_1)}\in\rm C_K$.
\subsection{Case $n=3$: Abel differential equation}
We consider a differential equation
\begin{equation}
\label{pol1}
y^\prime=P_3y^3+P_2y^2+P_1y+P_0,\quad(P_i)_{0\leqslant i\leqslant 3}\in K,\;P_3\not=0.
\end{equation}
Let $y_i$ be a particular solution of $\eqref{pol1}$ in some differential field extension of $K$ and $y$ be the general solution then
\begin{equation}
\label{pol2}
d\log(y-y_i):=\dfrac{\delta(y-y_i)}{y-y_i}=P_3(y^2+yy_i+y_i^2)+P_2(y+y_i)+P_1.
\end{equation}
From this we deduce
\begin{equation}
\label{pol3}
d\log(y-y_1)-d\log(y-y_2)=(y_1-y_2)(P_3(y+y_1+y_2)+P_2)
\end{equation}
where $y_1$, $y_2$ are distinct.

By using $y_1$ and $y_3$ instead of $y_1$ and $y_2$, such that all three $(y_i)_{1\leqslant i\leqslant3}$ are distinct, we obtain
\begin{equation}
\label{pol4}
d\log(y-y_1)-d\log(y-y_3)=(y_1-y_3)(P_3(y+y_1+y_3)+P_2).
\end{equation}
Hence
\begin{equation}
\label{pol013}
\dfrac{1}{(y_1-y_2)}d\log\left(\dfrac{y-y_1}{y-y_2}\right)-\dfrac{1}{y_1-y_3}d\log\left(\dfrac{y-y_1}{y-y_3}\right)=P_3(y_2-y_3).
\end{equation}
Thus after multiplying \eqref{pol013} by $(y_1-y_2)(y_1-y_3)$ we get
\begin{equation}
\label{pol5}
\omega_{13}d\log\left(\dfrac{y-y_1}{y-y_2}\right)-\omega_{12}d\log\left(\dfrac{y-y_1}{y-y_3}\right)=P_3\omega_{12}\omega_{13}\omega_{23},
\end{equation}
where $\omega_{ik}=y_i-y_k$. Let us choose $y=y_4$ then $y=y_5$ in \eqref{pol5} such that $(y_i)_{1\leqslant i\leqslant5}$ are distinct, then take the difference. This gives
\begin{equation}
\label{pol014}
\omega_{13}d\log\left(\dfrac{(y_4-y_1)(y_5-y_2)}{(y_4-y_2)(y_5-y_1)}\right)-\omega_{12}d\log\left(\dfrac{(y_4-y_1)(y_5-y_3)}{(y_4-y_3)(y_5-y_1)}\right)=0;
\end{equation}
this is a formula parallel to \eqref{pol08}.

If in \eqref{pol013} we consider the set $\{y_2,y_3,y_4\}$ of distinct solutions to \eqref{pol1} instead of $\{y_1,y_2,y_3\}$, where we assume that $(y_i)_{1\leqslant i\leqslant4}$ are all distinct, we find
\begin{equation}
\label{pol019}
\dfrac{1}{(y_2-y_3)}d\log\left(\dfrac{y-y_2}{y-y_3}\right)-\dfrac{1}{y_2-y_4}d\log\left(\dfrac{y-y_2}{y-y_4}\right)=P_3(y_3-y_4).
\end{equation}
After taking the difference between \eqref{pol013} and \eqref{pol019} (we divide equation \eqref{pol013}, resp. equation \eqref{pol019} by $(y_2-y_3)$, resp. $(y_3-y_4)$ before taking the difference of the resulting equations) we obtain
\begin{equation}
\label{pol200}
\begin{split}
\dfrac{1}{(y_2-y_3)}\left[\dfrac{1}{(y_1-y_2)}d\log\left(\dfrac{y-y_1}{y-y_2}\right)-\dfrac{1}{y_1-y_3}d\log\left(\dfrac{y-y_1}{y-y_3}\right)\right]\\-\dfrac{1}{(y_3-y_4)}\left[\dfrac{1}{(y_2-y_3)}d\log\left(\dfrac{y-y_2}{y-y_3}\right)-\dfrac{1}{y_2-y_4}d\log\left(\dfrac{y-y_2}{y-y_4}\right)\right]=0.
\end{split}
\end{equation}
After multiplication by $(y_1-y_2)(y_1-y_3)(y_2-y_3)(y_3-y_4)(y_2-y_4)$ equation \eqref{pol200} becomes 
\begin{equation}
\label{pol201}
\begin{split}
&\omega_{13}\omega_{34}\omega_{24}d\log\left(\dfrac{y-y_1}{y-y_2}\right)-\omega_{12}\omega_{34}\omega_{24}d\log\left(\dfrac{y-y_1}{y-y_3}\right)-\omega_{12}\omega_{13}\omega_{24}d\log\left(\dfrac{y-y_2}{y-y_3}\right)\\&+\omega_{12}\omega_{13}\omega_{23}d\log\left(\dfrac{y-y_2}{y-y_4}\right)=0
\end{split}
\end{equation} 
which after simplification translates into
\begin{equation}
\label{pol6}
\begin{split}
&\omega_{23}\omega_{24}\omega_{34}d\log(y-y_1)-\omega_{13}\omega_{14}\omega_{34}d\log(y-y_2)+\\
&\omega_{12}\omega_{14}\omega_{24}d\log(y-y_3)-\omega_{12}\omega_{13}\omega_{23}d\log(y-y_4)=0.
\end{split}
\end{equation}
This is a pendant to formula \eqref{pol012}.

Now after eliminating $P_3$, $P_2$, $P_1$, $P_0$ between the identities
\begin{equation}
\label{pol7}
y_i^\prime=P_3y^3_i+P_2y_i^2+P_1y_i+P_0,\quad i=1,2,3,4,5
\end{equation}
involving $5$ distinct solutions to \eqref{pol1} in a differential field extension of $K$, that is we consider the vanishing of the following determinant
\begin{equation}
\label{pol789}
\left|\begin{array}{ccccc}y_1^\prime & y_1^3 & y_1^2 & y_1 & 1 \\y_2^\prime & y_2^3 & y_2^2 & y_2 & 1\\y_3^\prime & y_3^3 & y_3^2 & y_3 & 1 \\y_4^\prime & y_4^3 & y_4^2 & y_4 & 1 \\y_5^\prime & y_5^3 & y_5^2 & y_5 & 1\end{array}\right|=0,
\end{equation}
we straightforwardly find that \eqref{pol789} is equivalent to the equation \eqref{pol8}
\begin{equation}
\label{pol8}
\omega_{14}\omega_{15}\omega_{24}\omega_{25}\omega_{45}\delta\left(\dfrac{\omega_{13}}{\omega_{12}}\right)+\omega_{13}\omega_{14}\omega_{23}\omega_{24}\omega_{34}\delta\left(\dfrac{\omega_{15}}{\omega_{12}}\right)+
\omega_{13}\omega_{15}\omega_{23}\omega_{25}\omega_{53}\delta\left(\dfrac{\omega_{14}}{\omega_{12}}\right)=0.
\end{equation}
This is the counterpart of the formula \eqref{pol0015}. In fact by using \eqref{pol019} twice where we replace the vector $(y,y_2,y_3,y_4)$ firstly by the vector $(y_1,y_2,y_3,y_4)$ then by the vector $(y_1,y_2,y_3,y_5)$ we obtain after having taken the difference of the gotten expressions the equality
\begin{equation}
\label{lhlhjklhknkklmjml}
\dfrac{\frac{\delta(y_1-y_2)}{(y_1-y_2)}-\frac{\delta(y_1-y_3)}{(y_1-y_3)}}{(y_2-y_3)(y_3-y_4)}-\dfrac{\frac{\delta(y_1-y_2)}{y_1-y_2}-\frac{\delta(y_1-y_4)}{(y_1-y_4)}}{(y_2-y_4)(y_3-y_4)}-\dfrac{\frac{\delta(y_1-y_2)}{y_1-y_2}-\frac{\delta(y_1-y_3)}{(y_1-y_3)}}{(y_2-y_3)(y_3-y_5)}+\dfrac{\frac{\delta(y_1-y_2)}{y_1-y_2}-\frac{\delta(y_1-y_5)}{(y_1-y_5)}}{(y_2-y_3)(y_3-y_5)}=0.
\end{equation}
After dividing the identity \eqref{lhlhjklhknkklmjml} by $\omega_{12}$ and multiplying it by minus the common denominator of the fractions involved, and then dividing the whole resulting identity by $(\omega_{12}^2)$ we get \eqref{pol8}. To obtain \eqref{pol789} from \eqref{lhlhjklhknkklmjml} it suffices to isolate each $\delta(y_i)$ in \eqref{lhlhjklhknkklmjml} and recognize that we find in this case the expansion of the left hand side of \eqref{pol789} with respect to its first column.
\subsection{The general case}
We treat mainly the case $n=4$ the general case shall be clear from the formulas for $n=4$. 
We consider a differential equation
\begin{equation}
\label{pol0101}
y^\prime=P_4y^4+P_3y^3+P_2y^2+P_1y+P_0,\quad(P_i)_{0\leqslant i\leqslant 4}\in K,\;P_4\not=0.
\end{equation}
Let $y_i$ be a particular solution of $\eqref{pol0101}$ in a differential field extension of $K$ then
\begin{equation}
\label{pol0102}
y^\prime-y^\prime_i=P_4(y^3+y^2y_i+yy_i^2+y_i^3)(y-y_i)+P_3(y^2+yy_i+y_i^2)(y-y_i)+ P_2(y+y_i)(y-y_i)+ P_1(y-y_i).
\end{equation}
Thus for two distinct solutions $y_1$, $y_2$ in a differential field extension of $K$ and for $y$ the general solution we get
\begin{equation}
\label{pol0103}
d\log\left(\dfrac{y-y_1}{y-y_2}\right)=(y_1-y_2)\left[P_4(y^2+y(y_1+y_2)+y_1^2+y_1y_2+y_2^2)+P_3(y+y_1+y_2)+P_2\right].
\end{equation}
This gives
\begin{equation}
\label{pol0104}
\dfrac{1}{(y_1-y_2)}d\log\left(\dfrac{y-y_1}{y-y_2}\right)=P_4(y^2+y(y_1+y_2)+y_1^2+y_1y_2+y_2^2)+P_3(y+y_1+y_2)+P_2.
\end{equation}
Applying the same reasoning with distinct $y_1$, $y_3$ such that $(y_i)_{1\leqslant i\leqslant3}$ are distinct as a whole we find
\begin{equation}
\label{pol0105}
\dfrac{1}{(y_1-y_3)}d\log\left(\dfrac{y-y_1}{y-y_3}\right)=P_4(y^2+y(y_1+y_3)+y_1^2+y_1y_3+y_3^2)+P_3(y+y_1+y_3)+P_2.
\end{equation}
If we take the difference between equations \eqref{pol0104} and \eqref{pol0105} we obtain
\begin{equation}
\label{pol0106}
\dfrac{1}{(y_1-y_2)}d\log\left(\dfrac{y-y_1}{y-y_2}\right)-\dfrac{1}{(y_1-y_3)}d\log\left(\dfrac{y-y_1}{y-y_3}\right)=(y_2-y_3)\left[P_4(y+y_1+y_2+y_3)+P_3\right].
\end{equation}
Hence
\begin{equation}
\label{pol0107}
\dfrac{1}{(y_2-y_3)}\left[\dfrac{1}{(y_1-y_2)}d\log\left(\dfrac{y-y_1}{y-y_2}\right)-\dfrac{1}{(y_1-y_3)}d\log\left(\dfrac{y-y_1}{y-y_3}\right)\right]=P_4(y+y_1+y_2+y_3)+P_3.
\end{equation}
Now if we do the same that we did for $(y_i)_{1\leqslant i\leqslant3}$ but now for $(y_i)$, with $i\in\{1,2,4\}$, such that $(y_i)_{1\leqslant i\leqslant4}$ are distinct as a whole we get the formula analogous to \eqref{pol0107}
\begin{equation}
\label{pol0108}
\dfrac{1}{(y_2-y_4)}\left[\dfrac{1}{(y_1-y_2)}d\log\left(\dfrac{y-y_1}{y-y_2}\right)-\dfrac{1}{(y_1-y_4)}d\log\left(\dfrac{y-y_1}{y-y_4}\right)\right]=P_4(y+y_1+y_2+y_4)+P_3.
\end{equation}
Then the difference between \eqref{pol0107} and \eqref{pol0108} gives
\begin{equation}
\label{pol0109}
\begin{split}
\dfrac{1}{(y_2-y_3)}\left[\dfrac{1}{(y_1-y_2)}d\log\left(\dfrac{y-y_1}{y-y_2}\right)-\dfrac{1}{(y_1-y_3)}d\log\left(\dfrac{y-y_1}{y-y_3}\right)\right]\\-\dfrac{1}{(y_2-y_4)}\left[\dfrac{1}{(y_1-y_2)}d\log\left(\dfrac{y-y_1}{y-y_2}\right)-\dfrac{1}{(y_1-y_4)}d\log\left(\dfrac{y-y_1}{y-y_4}\right)\right]=P_4(y_3-y_4).
\end{split}
\end{equation}
This gives a formula similar to \eqref{pol013}. Let us pause for a moment and find the formulas analogous to \eqref{pol5}, \eqref{pol014}. In order to achieve that we multiply \eqref{pol0109} by $(y_1-y_2)(y_1-y_3)(y_1-y_4)(y_2-y_3)(y_2-y_4)$ to find, with previous notations
\begin{equation}
\label{pol0110}
\begin{split}
&\omega_{13}\omega_{14}\omega_{34}d\log\left(\dfrac{y-y_1}{y-y_2}\right)-\omega_{12}\omega_{14}\omega_{24}d\log\left(\dfrac{y-y_1}{y-y_3}\right)+\omega_{12}\omega_{13}\omega_{23}d\log\left(\dfrac{y-y_1}{y-y_4}\right)\\&=P_4\omega_{12}\omega_{13}\omega_{14}\omega_{23}\omega_{24}\omega_{34}.
\end{split}
\end{equation}
This is similar to \eqref{pol5}. Let us choose solutions $y=y_5$, $y_6$ in a differential field extension of $K$ such that $(y_i)_{1\leqslant i\leqslant6}$ are distinct and replace them into \eqref{pol0110}. Then we get
\begin{equation}
\label{pol0111}
\begin{split}
&\omega_{13}\omega_{14}\omega_{24}d\log\left(\dfrac{y_5-y_1}{y_5-y_2}\right)-\omega_{12}\omega_{14}\omega_{24}d\log\left(\dfrac{y_5-y_1}{y_5-y_3}\right)-\omega_{13}\omega_{14}\omega_{23}d\log\left(\dfrac{y_5-y_1}{y_5-y_2}\right)\\&+\omega_{12}\omega_{13}\omega_{23}d\log\left(\dfrac{y_5-y_1}{y_5-y_4}\right)=P_4\omega_{12}\omega_{13}\omega_{14}\omega_{23}\omega_{24}\omega_{34}
\end{split}
\end{equation}
and
\begin{equation}
\label{pol0112}
\begin{split}
&\omega_{13}\omega_{14}\omega_{24}d\log\left(\dfrac{y_6-y_1}{y_6-y_2}\right)-\omega_{12}\omega_{14}\omega_{24}d\log\left(\dfrac{y_6-y_1}{y_6-y_3}\right)-\omega_{13}\omega_{14}\omega_{23}d\log\left(\dfrac{y_6-y_1}{y_6-y_2}\right)\\&+\omega_{12}\omega_{13}\omega_{23}d\log\left(\dfrac{y_6-y_1}{y_6-y_4}\right)=P_4\omega_{12}\omega_{13}\omega_{14}\omega_{23}\omega_{24}\omega_{34}.
\end{split}
\end{equation}
Hence after taking the difference of the two equations \eqref{pol0111} and \eqref{pol0112} we discover the identity
\begin{equation}
\label{pol0113}
\begin{split}
&\omega_{13}\omega_{14}\omega_{34}d\log\left(\dfrac{(y_5-y_1)(y_6-y_2)}{(y_5-y_2)(y_6-y_1)}\right)-\omega_{12}\omega_{14}\omega_{24}d\log\left(\dfrac{(y_5-y_1)(y_6-y_3)}{(y_5-y_3)(y_6-y_1)}\right)\\
&+\omega_{12}\omega_{13}\omega_{23}d\log\left(\dfrac{(y_5-y_1)(y_6-y_4)}{(y_5-y_4)(y_6-y_1)}\right)=0,
\end{split}
\end{equation}
a formula similar to \eqref{pol014}. Let us in \eqref{pol0109} use the set $\{y_2,y_3,y_4,y_5\}$ rather than the set $\{y_1,y_2,y_3,y_4\}$ where we assume that $(y_i)_{1\leqslant i\leqslant5}$ are distinct. Then we find
\begin{equation}
\label{pol010999}
\begin{split}
\dfrac{1}{(y_3-y_4)}\left[\dfrac{1}{(y_2-y_3)}d\log\left(\dfrac{y-y_2}{y-y_3}\right)-\dfrac{1}{(y_2-y_4)}d\log\left(\dfrac{y-y_2}{y-y_4}\right)\right]\\-\dfrac{1}{(y_3-y_5)}\left[\dfrac{1}{(y_2-y_3)}d\log\left(\dfrac{y-y_2}{y-y_3}\right)-\dfrac{1}{(y_2-y_5)}d\log\left(\dfrac{y-y_2}{y-y_5}\right)\right]=P_4(y_4-y_5).
\end{split}
\end{equation}
Now we divide equation \eqref{pol0109} by $(y_3-y_4)$ and equation \eqref{pol010999} by $(y_4-y_5)$ then take the difference of the gotten expressions and then after that multiply the latter difference by $(y_1-y_2)(y_1-y_3)(y_1-y_4)(y_2-y_3)(y_2-y_4)(y_2-y_5)(y_3-y_4)(y_3-y_5)(y_4-y_5)$ to obtain
\begin{equation}
\label{polpol01}
\begin{split}
&\omega_{23}\omega_{24}\omega_{25}\omega_{34}\omega_{35}\omega_{45}d\log(y-y_1)-\omega_{13}\omega_{14}\omega_{15}\omega_{34}\omega_{35}\omega_{45}d\log(y-y_2)\\&+\omega_{12}\omega_{14}\omega_{15}\omega_{24}\omega_{25}\omega_{45}d\log(y-y_3)-\omega_{12}\omega_{13}\omega_{15}\omega_{23}\omega_{25}\omega_{35}d\log(y-y_4)\\
&+\omega_{12}\omega_{13}\omega_{14}\omega_{23}\omega_{24}\omega_{34}d\log(y-y_5)=0,
\end{split}
\end{equation}
a formula similar to \eqref{pol6}. We also have the determinantal formula 
\begin{equation}
\label{polpolpol1}
\left|\begin{array}{cccccc}\delta(y_1) & y_1^4 & y_1^3 & y_1^2 & y_1&1 \\\delta(y_2) & y_2^4 & y_2^3 & y_2^2 & y_2&1 \\\delta(y_3) & y_3^4 & y_3^3 & y_3^2 & y_3&1 \\\delta(y_4) & y_4^4 & y_4^3 & y_4^2 & y_4&1 \\\delta(y_5) & y_5^4 & y_5^3 & y_5^2 & y_5&1\\\delta(y_6) & y_6^4 & y_6^3 & y_6^2 & y_6&1\end{array}\right|=0
\end{equation}
analogous to \eqref{pol8}, for $(y_i)_{1\leqslant i\leqslant 6}$ distinct fixed solutions in a differential field extension of $K$. More generally we have similarly as above
\begin{theorem}
\label{curious}
Consider the differential equation 
\begin{equation}
\label{pol101102}
y^\prime=P_ny^n+P_{n-1}y^{n-1}+\ldots+P_1y+P_0,\quad (P_i)_{0\leqslant i\leqslant n}\in K,\;n\geqslant2,\;P_n\not=0
\end{equation}
over the differential field $K$ of characteristic zero with $C_K=\ker\delta$ algebraically closed. Let $(y_i)_{1\leqslant i\leqslant n+2}$ be $(n+2)$ given distinct solutions of \eqref{pol101102} in a differential field extension of $K$ and $y$ be its general solution. Then the following formulas hold
\begin{equation}
\label{pol101103}
\displaystyle\sum_{i=2}^n(-1)^i\left(\prod_{\substack{1\leqslant k<l\leqslant n,\\k,\,l\;\not=i}}\omega_{kl}\right)d\log\left(\dfrac{y-y_1}{y-y_i}\right)=P_n\prod_{1\leqslant \alpha<\beta\leqslant n}\omega_{\alpha\beta},\quad\omega_{\alpha\beta}=y_\alpha-y_\beta,\;\omega_{kl}=y_k-y_l,
\end{equation}
\begin{equation}
\label{pol101104}
\displaystyle\sum_{i=2}^n(-1)^i\left(\prod_{\substack{1\leqslant k<l\leqslant n,\\k,\,l\;\not=i}}\omega_{kl}\right)d\log\left(\dfrac{(y_{n+1}-y_1)(y_{n+2}-y_i)}{(y_{n+1}-y_i)(y_{n+2}-y_1)}\right)=0,\quad\omega_{kl}=y_k-y_l,
\end{equation}
\begin{equation}
\label{pol101105}
\displaystyle\sum_{i=1}^{n+1}(-1)^{i-1}\left(\prod_{\substack{1\leqslant k<l\leqslant n,\\k,\,l\;\not=i}}\omega_{kl}\right)d\log\left(y-y_i\right)=0,\quad\omega_{kl}=y_k-y_l.
\end{equation}
Furthermore we have the following determinantal formula between any distinct fixed solutions $(y_i)_{1\leqslant i\leqslant n+2}$ in a differential field extension of $K$
\begin{equation}
\left|\begin{array}{ccccccc}\delta(y_1) & y^{n}_1 & y_1^{n-1} & \ldots & \ldots & y_1 & 1 \\\delta(y_2) & y^{n}_2 & y_2^{n-1} & \ldots & \ldots & y_2 & 1 \\\delta(y_3) & y^{n}_3 & y_3^{n-1} & \ldots & \ldots & y_3 & 1 \\\ldots & \ldots & \ldots & \ldots & \ldots & \ldots & \ldots \\\ldots & \ldots & \ldots & \ldots & \ldots & \ldots & \ldots \\ \delta(y_{n+1}) & y^{n}_{n+1} & y_{n+1}^{n-1} & \ldots & \ldots & y_{n+1} & 1  \\\delta(y_n) & y^{n}_{n+2} & y_{n+2}^{n-1} & \ldots & \ldots & y_{n+2} & 1 \end{array}\right|=0.
\end{equation}
\end{theorem}
\begin{proof}
The proof is constructive and follows the steps explained above.
\end{proof}
\section{Obstruction to the constancy of cross-ratios of solutions of non-Riccati polynomial differential equations}
\label{obstructionto}
In this section we explain why when $n\geq3$, the cross-ratio of $4$ solutions to an equation \eqref{pol9} below is not constant in general. We achieve this by exhibiting an "obstruction" which prevents that from happening. 

We start by considering the differential equation
\begin{equation}
\label{pol9}
y^\prime=P_ny^n+P_{n-1}y^{n-1}+\ldots+P_1y+P_0,\quad (P_i)_{0\leqslant i\leqslant n}\in K,\; n\geqslant2.\end{equation}
Let $y_1$, $y_2$, $y_3$ be three distinct particular solutions to it. Then by linear algebra
\begin{equation}
\label{pol19}
\left|\begin{array}{ccc}\delta(y_1)-P_ny_1^n-\ldots-P_2y_1^2 & y_1 & 1 \\\delta(y_2)-P_ny_2^n-\ldots-P_2y_2^2 & y_2 & 1 \\\delta(y_3)-P_ny_3^n-\ldots-P_2y_3^2 & y_3 & 1\end{array}\right|=0. 
\end{equation}
We set for $s\geq2$
\begin{equation}
\label{pol20}
\left|\begin{array}{ccc}y_1^s & y_1 & 1 \\y_2^s & y_2 & 1 \\y_3^s & y_3 & 1\end{array}\right|=:\left|\begin{array}{ccc}y_1^2 & y_1 & 1 \\y_2^2 & y_2 & 1 \\y_3^2 & y_3 & 1\end{array}\right|A_{s-2}(y_1,y_2,y_3)
\end{equation}
where $A_{s-2}$ is a homogeneous symmetric polynomial function of degree $s-2$ of the $y_i$. In fact we have
\begin{equation}
\label{pol210000}
A_{s-2}=\displaystyle\sum_{\alpha\in\N^3,\;|\alpha|=s-2}y_1^{\alpha_1}y_2^{\alpha_2}y_3^{\alpha_3}.
\end{equation}
This follows from the theory of Schur polynomials, see \cite[§3, p. 40, eq. 3.1]{macdonald1995}. One may also express $A_{s-2}$ in terms of elementary symmetric functions, cf. \cite[formula 3.5, p. 41]{macdonald1995} where it is shown that
\begin{equation}
\label{pol21}
A_{s-2}=\left|\begin{array}{cccccc}c_1 & c_2 & c_3 & 0 & \ldots&0 \\1 & c_1 & c_2 & c_3 & \ldots&0 \\0 & 1 & c_1 & c_2 & \ldots&0\\\ldots& \ldots & \ldots & \ldots & \ldots&\ldots \\0 & 0 & 0 & 0 & \ldots&c_1\end{array}\right|
\end{equation}
where $c_1$ is the sum of the $(y_i)_{1\leqslant i\leqslant 3}$, $c_2$ is the sum of the products $2$ by $2$ of the $y_i$, i.e. $y_1y_2+y_1y_3+y_2y_3$, $c_3=y_1y_2y_3$. The equation \eqref{pol21} for $A_{s-2}$ is a determinant of order $s-2$. By developing the determinant \eqref{pol19} and by using \eqref{pol20} we have the formula
\begin{equation}
\label{pol00022}
\begin{split}
(y_3-y_1)^2\delta\left(\dfrac{y_2-y_1}{y_3-y_1}\right)&=(P_nA_{n-2}+P_{n-1}A_{n-3}+\ldots+P_2)V(y_1,y_2,y_3)\\
&=(P_nA_{n-2}+P_{n-1}A_{n-3}+\ldots+P_2)(y_2-y_1)(y_3-y_1)(y_2-y_3)
\end{split}
\end{equation}
where 
$$V(y_1,y_2,y_3)=\left|\begin{array}{ccc}y_1^2 & y_1 & 1 \\y_2^2 & y_2 & 1 \\y_3^2 & y_3 & 1\end{array}\right|.$$
This gives
\begin{equation}
\label{pol22}
d\log\dfrac{y_1-y_2}{y_1-y_3}=(P_nA_{n-2}+P_{n-1}A_{n-3}+\ldots+P_2)(y_2-y_3).
\end{equation}
For $n=2$, \eqref{pol22} exactly gives \eqref{pol09} when the latter is applied to $y_1$ which is distinct from $y_2$ and $y_3$.
Let us introduce $B_{s-3}$ for $s\geq3$ as follows
\begin{equation}
\label{pol00025}
\left|\begin{array}{cccc}y_1^s & y_1^2 & y_1 & 1 \\y_2^s & y_2^2 & y_2 & 1 \\y_3^s & y_3^2 & y_3 & 1 \\y_4^s & y_4^2 & y_4 & 1\end{array}\right|=:\left|\begin{array}{cccc}y_1^3 & y_1^2 & y_1 & 1 \\y_2^3 & y_2^2 & y_2 & 1 \\y_3^3 & y_3^2 & y_3 & 1 \\y_4^3 & y_4^2 & y_4 & 1\end{array}\right|B_{s-3}.
\end{equation}
These are the expressions analogous to the $A_{s-2}$, but which are constructed with the four particular distinct solutions $y_1$, $y_2$, $y_3$, $y_4$ instead of the three distinct solutions $y_1$, $y_2$, $y_3$. As before with $A_{s-2}$ we have by \cite[§3, p. 40, eq. 3.1]{macdonald1995}
\begin{equation}
\label{pol000251}
B_{s-3}=\displaystyle\sum_{\alpha\in\N^4,\;|\alpha|=s-3}y_1^{\alpha_1}y_2^{\alpha_2}y_3^{\alpha_3}y^{\alpha_4}_4.
\end{equation}
$B_{s-3}$ is a homogeneous symmetric polynomial in $(y_i)_{1\leq i\leq4}$ and of degree $s-3$. We set $$V(y_1,y_2,y_3,y_4)=\left|\begin{array}{cccc}y_1^3 & y_1^2 & y_1 & 1 \\y_2^3 & y_2^2 & y_2 & 1 \\y_3^3 & y_3^2 & y_3 & 1 \\y_4^3 & y_4^2 & y_4 & 1\end{array}\right|=(y_2-y_1)(y_3-y_1)(y_4-y_1)(y_3-y_2)(y_4-y_2)(y_4-y_3).$$
Then if instead of three particular distinct solutions of \eqref{pol9} we consider four particular distinct solutions to \eqref{pol9}, we have after eliminating $P_2$, $P_1$, $P_0$, instead of \eqref{pol19}, a $4$ by $4$ determinant
\begin{equation}
\label{pol23}
\left|\begin{array}{cccc}\delta(y_1)-P_ny_1^n-\ldots-P_3y_1^3&y_1^2& y_1 & 1 \\\delta(y_2)-P_ny_2^n-\ldots-P_3y_2^3&y_2^2 & y_2 & 1 \\\delta(y_3)-P_ny_3^n-\ldots-P_3y_3^3 &y_3^2& y_3 & 1\\\delta(y_4)-P_ny_4^n-\ldots-P_3y_4^3&y_4^2& y_4 & 1\end{array}\right|=0.
\end{equation}
Hence after developing the determinant \eqref{pol23} we obtain for $4$ given distinct solutions of \eqref{pol9} the next identity
\begin{equation}
\label{pol2400}
(y_1-y_4)^2(y_3-y_2)^2\delta\left(\dfrac{(y_1-y_2)(y_3-y_4)}{(y_1-y_4)(y_3-y_2)}\right)=(P_{n}B_{n-3}+P_{n-1}B_{n-4}+\ldots+P_3)V(y_1,y_2,y_3,y_4).
\end{equation}
This gives
\begin{equation}
\label{pol24}
d\log\lambda=d\log\left(\dfrac{(y_1-y_2)(y_3-y_4)}{(y_1-y_4)(y_3-y_2)}\right)=(P_{n}B_{n-3}+P_{n-1}B_{n-4}+\ldots+P_3)(y_4-y_2)(y_3-y_1)
\end{equation}
where 
\begin{equation}
\label{pol25}
\lambda:=\dfrac{(y_1-y_2)(y_3-y_4)}{(y_1-y_4)(y_3-y_2)}
\end{equation}
is the cross-ratio.

For the Riccati equation we have $n=2$ and the $P_i$ for which $i\geqslant3$ are all zero. Thus from \eqref{pol24} we rediscover the fact, \eqref{pol08},\eqref{pol0015} that the cross-ratio of any four distinct solutions to a Riccati equation is constant: lies in $\ker\delta$ and one may as explained before, arrange that it lies in $C_K$, by choosing the $(y_i)_{1\leq i\leq4}$ in a Picard-Vessiot extension of $K$.
\begin{theorem}
\label{obstruction}
The cross-ratio of four distinct solutions $(y_i)_{1\leqslant i\leqslant 4}$ of a differential equation 
$$y^\prime=P_ny^n+P_{n-1}y^{n-1}+\ldots+P_1y+P_0,\quad (P_i)_{0\leqslant i\leqslant n}\in K,\;n\geqslant3,$$
is in general not constant and it does not belong in general to the kernel of the natural derivation, still denoted by $\delta$, on $K(y_1,y_2,y_3,y_4)$ and which extends the given derivation $\delta$ on the base differential field $K$: $y_i^\prime=P_ny^n_i+P_{n-1}y^{n-1}_i+\ldots+P_1y_i+P_0$ for $i=1$, $2$, $3$, $4$. Furthermore the obstruction for the cross-ratio of the $y_i$ to lie in $\ker\delta$ is given by equation \eqref{pol24}
$$d\log\lambda=d\log\left(\dfrac{(y_1-y_2)(y_3-y_4)}{(y_1-y_4)(y_3-y_2)}\right)=(P_{n}B_{n-3}+P_{n-1}B_{n-4}+\ldots+P_3)(y_4-y_2)(y_3-y_1)$$
with 
$$\lambda:=\dfrac{(y_1-y_2)(y_3-y_4)}{(y_1-y_4)(y_3-y_2)}.$$
We also have the identity \eqref{pol22}
$$d\log\dfrac{y_1-y_2}{y_1-y_3}=(P_nA_{n-2}+P_{n-1}A_{n-3}+\ldots+P_2)(y_2-y_3)$$
which holds for any three given solutions to the differential equation
$$y^\prime=P_ny^n+P_{n-1}y^{n-1}+\ldots+P_1y+P_0,\quad (P_i)_{0\leqslant i\leqslant n}\in K,\;n\geqslant2.$$
Assume that the field of constants $C_K$ of $K$ is algebraically closed and consider four distinct solutions $(y_i)_{1\leqslant i\leqslant4}$ to 
$$y^\prime=P_ny^n+P_{n-1}y^{n-1}+\ldots+P_1y+P_0,\quad (P_i)_{0\leqslant i\leqslant n}\in K,\;n\geqslant3,\;P_n\not=0$$
which lie in a same differential field extension $L$ of $K$ with the same field of constants as $K$. Then if the $(y_i)_{1\leqslant i\leqslant4}$ are algebraically independent over $K$ then there does not exist $G_0\not=0$ and $G_1$, both algebraic over $K$ such that
\begin{equation}
\label{impossible}
d\log\lambda-d\log(G_0):=d\log\left(\dfrac{(y_1-y_2)(y_3-y_4)}{(y_1-y_4)(y_3-y_2)}\right)-d\log(G_0)=G_1^\prime
\end{equation}
with $\lambda$ as before.

Similarly let $K$ be as before with $C_K$ algebraically closed. Suppose that we are given three distinct solutions $(y_i)_{1\leqslant i\leqslant3}$ to 
$$y^\prime=P_ny^n+P_{n-1}y^{n-1}+\ldots+P_1y+P_0,\quad (P_i)_{0\leqslant i\leqslant n}\in K,\;n\geqslant2,\;P_n\not=0.$$
Then if the $(y_i)_{1\leqslant i\leqslant3}$ belong to a same differential field extension $L_1$ of $K$ with the same field of constants as $K$ and the $(y_i)_{1\leqslant i\leqslant3}$ are algebraically independent over $K$, then there does not exist $H_0\not=0$ and $H_1$ both algebraic over $K$ such that
\begin{equation}
\label{impossible1}
d\log\dfrac{y_1-y_2}{y_1-y_3}-d\log(H_0)=H_1^\prime.
\end{equation}
\end{theorem}
\begin{proof}
By what we have already explained only the two last assertions of the theorem remain to be proved. To see their truths under the hypotheses above we reason by contradiction and use \eqref{pol24}. Then given four solutions $(y_i)_{1\leqslant i\leqslant4}$ to 
$$y^\prime=P_ny^n+P_{n-1}y^{n-1}+\ldots+P_1y+P_0,\quad (P_i)_{0\leqslant i\leqslant n}\in K,\;n\geqslant3$$
which lie in a same differential field extension $L$ of $K$ with the same field of constants as $K$ and which are such that $(y_i)_{1\leqslant i\leqslant4}$ algebraically independent of $K$, then if we had the existence of $H_0$ and $H_1$ as above in \eqref{impossible}, then via \eqref{pol24} we would have a non-trivial polynomial algebraic relation between the $(y_i)_{1\leqslant i\leqslant4}$, by construction of the $B_{s-3}$, $3\leqslant s\leqslant n$. This would constitute a contradiction. A similar reasoning gives the last part of the theorem.
\end{proof}
\begin{remark}
The impossibility of an identity such as \eqref{impossible} should be read in plain text as there does not exist $G_0\not=0$ and $G_1$, both algebraic over $K$ such that
$$"\left(\dfrac{(y_1-y_2)(y_3-y_4)}{(y_1-y_4)(y_3-y_2)}\right)=G_0e^{G_1}."$$
Likewise the impossibility of an identity such as \eqref{impossible1} should be read in plain text as there does not exist $H_0\not=0$ and $H_1$, both algebraic over $K$ such that
$$"\dfrac{y_1-y_2}{y_1-y_3}=H_0e^{H_1}".$$
\end{remark}
\begin{remark}
We have seen that the cross-ratio of four distinct solutions to an equation 
$$y^\prime=P_ny^n+P_{n-1}y^{n-1}+\ldots+P_1y+P_0,\quad (P_i)_{0\leqslant i\leqslant n}\in K,\;n\geqslant3$$
is in general not constant. More dramatically for the innocent looking differential equation such as
$$y^\prime=y^3-y^2$$
over a characteristic zero algebraically closed field of constants $C$, it has been shown that any $m$ non-constant distinct solutions of it are algebraically independent over $C(x)$. In particular no cross-ratio of four distinct solutions of it can be an element of $C$. For more on this line of thought we refer to \cite[cor, p. 532-533]{rosenlicht}, \cite[lem. 7.2, prop. 7.3]{put2022}, \cite[§8, p. 160]{kolchin1974}.
\end{remark}
\section{Construction of a polynomial differential equation of Generalized Abel-Riccati type satisfied by the roots of a given separable polynomial}
\label{algebraic}
Let us consider now a separable polynomial
\begin{equation}
\label{pol26}
P_\varphi(z):=z^{n+1}+\varphi_1(x)z^n+\varphi_2(x)z^{n-1}+\ldots+\varphi_{n+1}(x),\;(\varphi_i)_{1\leqslant i\leqslant n+1}\in K
\end{equation}
where $(K,\frac{d}{dx})$ is some differential field of meromorphic functions on some domain of $\C$ with algebraically closed field of constants $C_K$. We want to construct a differential equation of the form \eqref{pol9} with coefficients in $K$ and satisfied by all the roots of \eqref{pol26}. Firstly remark that the field $L$ generated over $K$ by the different roots of $P_\varphi(z)$ is in a unique way a differential field and it has no new constants, i.e. $C_L=C_K$, since each root of $P_\varphi(z)$ is algebraic over $K$ and $C_K$ is algebraically closed (this follows from an iteration of \cite[Th. 6.2.6]{chambert2005}). Let $\xi_p(\varphi_1,\varphi_2,\ldots,\varphi_{n+1},x)$, $1\leqslant p\leqslant n+1$ be any root of \eqref{pol26} in a differential field extension of $K$, where we consider the $\varphi_i$ as variables. Then we have the relations
\begin{equation}
\label{pol27}
\sum_{j=1}^{n+1}\alpha_{ij}\dfrac{\partial \xi_p}{\partial \varphi_j}+\xi_p^i=0,\quad i=0,1,2,\ldots,n;\;p=1,2,\ldots,n+1
\end{equation}
where
\begin{equation}
\label{brioschi}
\alpha_{ij}=S_{i+j-1}+\varphi_1S_{i+j-2}+\ldots+\varphi_{j-1}S_i, \quad i=0,1,2,\ldots,n;\;j=1,2,\ldots,n+1,
\end{equation}
and where $S_k$ is the sum of the $k^{th}$ powers of the roots of \eqref{pol26}.
\begin{lemma}[Proof of \eqref{pol27}]
\label{lemrefAE}
Let $(K,\delta=\frac{d}{dx})$ be a differential field of meromorphic functions on some domain of $\C$ with algebraically closed field of constants $C_K=\C$. Let 
\begin{equation}
\label{pol2600000008800808}
P_\varphi(z)=z^{n+1}+\varphi_1(x)z^n+\varphi_2(x)z^{n-1}+\ldots+\varphi_{n+1}(x),\;(\varphi_i)_{1\leqslant i\leqslant n+1}\in K,\;n\geqslant2
\end{equation}
be a separable polynomial over $K$. We consider the coefficients $\varphi_1$, $\varphi_2$, $\ldots,$ $\varphi_{n+1}$ as independent variables. Then as such let $\xi_p(x,\varphi_1,\varphi_2,\ldots,\varphi_{n+1})$ be any root of \eqref{pol2600000008800808}. Then we have the following system of partial differential equations
\begin{equation}
\label{pol2908007}
\sum_{j=1}^{n+1}\alpha_{ij}\dfrac{\partial \xi_p}{\partial \varphi_j}+\xi_p^i=0,\quad i=0,1,2,\ldots,n;\;p=1,2,\ldots,n+1
\end{equation}
where
\begin{equation}
\label{brioschizeyz}
\alpha_{ij}=S_{i+j-1}+\varphi_1S_{i+j-2}+\ldots+\varphi_{j-1}S_i, \quad i=0,1,2,\ldots,n;\;j=1,2,\ldots,n+1,
\end{equation}
and where $S_k$ is the sum of the $k^{th}$ powers of the roots of \eqref{pol2600000008800808}.
\end{lemma}
\begin{proof}
Let $\xi_1$ be any root of $P_\varphi(z)$. Then for any $i\in\{0,1,\ldots,n\}$ we have identically 
\begin{equation}
\label{LHGJKG}
z^i\dfrac{P_\varphi(z)}{(z-\xi_1)}=z^{i+n}+(\xi_1+\varphi_1)z^{i+n-1}+(\xi_1^2+\varphi_1\xi_1+\varphi_2)z^{i+n-2}+\ldots+(\xi_1^{n}+\varphi_1\xi_1^{n-1}+\ldots+\varphi_n)z^i.
\end{equation}
This is true because $\xi_1$ is a root of $P_\varphi(z)$ therefore the multiplication the quotient of the division of $P_\varphi(z)$ by $(z-\xi_1)$ by $z^i$ gives \eqref{LHGJKG}. 

Let us set in \eqref{LHGJKG} $z$ equal successively to the roots $\xi_1$, $\xi_2$, $\ldots,$ $\xi_{n+1}$ of $P_\varphi(z)$, and take the sum of resulting expressions to get
\begin{equation}
\label{oiuuhiuinibn}
\xi_1^iP^\prime_\varphi(\xi_1)=s_{i+n}+(\xi_1+\varphi_1)s_{i+n-1}+(\xi_1^2+\varphi_1\xi_1+\varphi_2)s_{i+n-2}+\ldots+(\xi_1^n+\varphi_1\xi_1^{n-1}+\ldots+\varphi_n)s_i.
\end{equation}
Since $\xi_1$ is arbitrarily chosen it also holds that
\begin{equation}
\label{lblbnnjklhh}
\xi_p^iP_\varphi^\prime(\xi_p)=\alpha_{i,1}\xi^n_p+\alpha_{i,2}\xi^{n-1}_p+\ldots+\alpha_{i,n+1},\;\;i\in\{0,1,\ldots,n\}
\end{equation}
with $\alpha_{i,j}$ given by \eqref{brioschizeyz}. Besides by taking the derivation of $P_\varphi(\xi_p)=0$ with respect to $\varphi_i$ we obtain
\begin{equation}
\label{mnlnhjkb}
\begin{split}
P_\varphi^\prime(\xi_p)\dfrac{\partial\xi_p}{\partial\varphi_1}+\xi_p^n=0, P_\varphi^\prime(\xi_p)\dfrac{\partial\xi_p}{\partial\varphi_2}+\xi_p^{n-1}=0,\ldots,P_\varphi^\prime(\xi_p)\dfrac{\partial\xi_p}{\partial\varphi_{n+1}}+1=0
\end{split}
\end{equation}
Let us multiply the first equation in \eqref{mnlnhjkb} by $\alpha_{i,1}$, resp. the second equation in \eqref{mnlnhjkb} by $\alpha_{i,2}$, $\ldots$, resp. the last equation in \eqref{mnlnhjkb} by $\alpha_{i,n+1}$ and take the sum of the resulting equations member by member. Then if we use \eqref{lblbnnjklhh} we find 
\begin{equation}
\label{ljnlllnunjln}
P_\varphi^\prime(\xi_p)\left(\sum_{j=1}^{n+1}\alpha_{ij}\dfrac{\partial \xi_p}{\partial \varphi_j}+\xi_p^i\right)=0,\quad i=0,1,2,\ldots,n;\;p=1,2,\ldots,n+1.
\end{equation}
But $P_\varphi^\prime(\xi_p)\not=0$ for $p=1,2,\ldots,n+1$, since $P_\varphi(z)$ is separable by hypothesis. Hence 
\begin{equation}
\label{ljnlllnunhfghhvjln}
\sum_{j=1}^{n+1}\alpha_{ij}\dfrac{\partial \xi_p}{\partial \varphi_j}+\xi_p^i=0,\quad i=0,1,2,\ldots,n;\;p=1,2,\ldots,n+1.
\end{equation}
This proves our claim.
\end{proof}
\begin{remark}
As stated it may happen that $P_n\equiv0$ if the $\varphi_i$ are not well chosen but it is easy to arrange so that $P_n\not=0$, by requiring the non-vanishing of a certain determinant in the $\varphi_i$ and their derivatives. Furthermore linear pdes satisfied by roots of polynomials have been previously studied by Mayr, \cite{mayr} and in the theory of GKZ systems, \cite{sturmfels2000}. In fact one may show that the roots of a separable polynomial
$$b_{0}z^n+b_1z^{n-1}+b_2z^{n-2}+\ldots+b_n$$
satisfy a GKZ hypergeometric system of pdes with respect to the coefficients $b_i$ whose solution space is finite dimensional over $\C$, cf. \cite{mayr, sturmfels2000}.
\end{remark}

Before going further let us remind that each $S_k$ belongs to $\C\left[\varphi_1,\varphi_2,\ldots,\varphi_{n+1}\right]$ by the Newton identities. Also remember that we are looking for a polynomial differential equation of the form \eqref{pol9} that $\xi_p$ may satisfy. Thus we derive $\xi_p=\xi_p(x,\varphi_1,\varphi_2,\ldots,\varphi_n)$ with respect to $x$ and obtain the identical relation
\begin{equation}
\label{pol28}
\sum_{j=1}^{n+1}\varphi^\prime_j(x)\dfrac{\partial \xi_p}{\partial \varphi_j}-\dfrac{d\xi_p}{dx}=0,
\end{equation}
with $\varphi^\prime_j=\frac{d\varphi_j}{dx}$. If we eliminate the partial derivatives of $\xi_p$ with respect to $\varphi_j$, between \eqref{pol27} and \eqref{pol28}, we obtain
\begin{equation}
\label{pol29}
\left|\begin{array}{cccc}\frac{d\xi_p}{dx} & \varphi_1^\prime(x) & \ldots & \varphi_{n+1}^\prime(x) \\-1 & \alpha_{0,1} & \ldots & \alpha_{0,n+1} \\-\xi_p & \alpha_{1,1} & \ldots & \alpha_{1,n+1} \\\ldots &\ldots & \ldots & \ldots\\-\xi_p^n & \alpha_{n,1} & \ldots & \alpha_{n,n+1}\end{array}\right|=0.\end{equation}
The cofactor associated to $\frac{d\xi_p}{dx}$ is the discriminant of the polynomial \eqref{pol26} which is non zero since $P_\varphi(z)$ is separable. In fact let us recall the following relation between the Vandermonde matrix on the roots of \eqref{pol26}, its transpose, the Hankel matrix of the sum of powers of the roots. This will allow us to show that the discriminant of \eqref{pol26} is equal to the cofactor of $\frac{d\xi_p}{dx}$. We have
\begin{equation}
\label{oopkklhjjkjkjklhj}
\left(\begin{array}{ccccc}1 & 1 & \ldots & \ldots & 1 \\\xi_1 & \xi_2 & \ldots & \ldots & \xi_{n+1} \\\xi_1^2 & \xi_2^2 & \ldots & \ldots & \xi_{n+1}^2 \\\ldots & \ldots & \ldots & \ldots & \ldots \\\xi_1^{n} & \xi_2^n & \ldots & \ldots & \xi_{n+1}^n\end{array}\right)\left(\begin{array}{ccccc}1 & \xi_1 & \xi_1^2 & \ldots & \xi_1^n \\1 & \xi_2 & \xi_2^2 & \ldots & \xi_2^n \\\vdots & \vdots & \vdots & \vdots & \vdots \\\vdots & \vdots & \vdots & \vdots & \vdots \\1 & \xi_{n+1} & \xi_{n+1}^2 & \ldots & \xi_{n+1}^n\end{array}\right)=\left(\begin{array}{ccccc}S_0 & S_1 & S_2 & \ldots & S_n \\S_1 & S_2 & S_3 & \ldots & S_{n+1} \\\vdots & \vdots & \vdots & \vdots& \vdots \\S_{n-1} & S_n & S_{n+1} & \ldots & S_{2n-1} \\S_{n} & S_{n+1} & S_{n+2} & \ldots & S_{2n}\end{array}\right).
\end{equation}
Now on the one hand the square of the determinant of the Vandermonde formed from the $(n+1)$ roots of \eqref{pol26} is exactly the discriminant of \eqref{pol26} while on the other hand one readily checks using the formulas \eqref{brioschi} for the $\alpha_{i,j}$ that the determinant of the matrix on the right hand side of \eqref{oopkklhjjkjkjklhj} is exactly equal to the cofactor of $\frac{d\xi_p}{dx}$ in \eqref{pol29}. Hence our claim about this cofactor and our deduction that it is different from zero. Therefore since $\xi_p$ is arbitrary all the roots of the equation \eqref{pol26} are also solution to the differential equation
\begin{equation}
\label{pol2900}
\left|\begin{array}{cccc}\frac{dy}{dx} & \varphi_1^\prime(x) & \ldots & \varphi_{n+1}^\prime(x) \\-1 & \alpha_{0,1} & \ldots & \alpha_{0,n+1} \\-y & \alpha_{1,1} & \ldots & \alpha_{1,n+1} \\\ldots &\ldots & \ldots & \ldots\\-y^n & \alpha_{n,1} & \ldots & \alpha_{n,n+1}\end{array}\right|=0\end{equation}
which we may also write in the form \eqref{pol9}
$$y^\prime=P_ny^n+P_{n-1}y^{n-1}+\ldots+P_1y+P_0,\;\;\;P_j\in K,\;^\prime=\frac{d}{dx}.$$
Hence we have the
\begin{theorem}
\label{brm}
Let $(K,\delta=\frac{d}{dx})$ be a differential field of meromorphic functions on some domain of $\C$ with algebraically closed field of constants $C_K=\C$. Let 
\begin{equation}
\label{pol260000000}
P_\varphi(z)=z^{n+1}+\varphi_1(x)z^n+\varphi_2(x)z^{n-1}+\ldots+\varphi_{n+1}(x),\;(\varphi_i)_{1\leqslant i\leqslant n+1}\in K,\;n\geqslant2
\end{equation}
be a separable polynomial over $K$. Then there exists a differential of the form
\begin{equation}
\label{pol98976797}
y^\prime=P_ny^n+P_{n-1}y^{n-1}+\ldots+P_1y+P_0,\quad (P_i)_{0\leqslant i\leqslant n}\in K\end{equation}
which is satisfied by all the roots of $P_\varphi(z)$. Moreover if there exists another such differential equation of the same form
\begin{equation}
\label{gghgkhgkjhnghjg}
y^\prime=Q_ny^n+Q_{n-1}y^{n-1}+\ldots+Q_1y+Q_0,\quad (Q_i)_{0\leqslant i\leqslant n}\in K,
\end{equation}
which is satisfied by all the roots of $P_\varphi(z)$ then the two differential equation \eqref{pol260000000}, \eqref{pol98976797} are identical. In other words the differential equation \eqref{pol98976797} is the unique differential equation of the form 
$$y^\prime=R_ny^n+R_{n-1}y^{n-1}+\ldots+R_1y+R_0,\quad (R_i)_{0\leqslant i\leqslant n}\in K$$
which is satisfied by all the roots of $P_\varphi(z)$.
\end{theorem}
\begin{proof}
We have already shown constructively how to obtain a differential equation of the form \eqref{pol98976797} starting from the datum of $P_\varphi(z)$, \eqref{pol260000000}. Let us see why $P_\varphi(z)$ must be unique. Suppose all the roots of $P_\varphi(z)$ satisfy both equations \eqref{pol260000000}, \eqref{pol98976797} . Then for $\xi$ an arbitrary such root we have
\begin{equation}
\label{GHKLGHHG}
\begin{split}
\dfrac{d\xi}{dx}=P_n\xi^n+P_{n-1}\xi^{n-1}+\ldots+P_1\xi+P_0=Q_n\xi^n+Q_{n-1}\xi^{n-1}+\ldots+Q_1\xi+Q_0.
\end{split}
\end{equation}
Hence the polynomial 
$$(P_n-Q_n)z^n+(P_{n-1}-Q_{n-1})z^{n-1}+\ldots+(P_0-Q_0)$$
admits $n+1$ distinct roots, and must be the zero polynomial. Thus 
$$P_0=Q_0,\;P_1=Q_1,\ldots,P_n=Q_n.$$
\end{proof}
\begin{example}
\label{onetoone}
Let us look at an example, the case $n=3$. In this case the cross-ratio $\lambda$ of the four roots of the separable degree four polynomial
\begin{equation}
\label{pol30}
P_\varphi(z)=z^{4}+\varphi_1(x)z^3+\varphi_2(x)z^{2}+\varphi_3(x)z+\varphi_{4}(x),\;(\varphi_i)_{1\leqslant i\leqslant 4}\in K
\end{equation}
satisfies the relation
\begin{equation}
\label{pol31}
d\log\lambda=P_3(y_1-y_2)(y_3-y_4),
\end{equation}
with
\begin{equation*}
\label{jibreel}
P_3=\dfrac{1}{\mathscr D}\left|\begin{array}{cccc}\varphi_1^\prime & \varphi_2^\prime & \varphi_3^\prime & \varphi_4^\prime \\4 & 3\varphi_1 & 2\varphi_2 & \varphi_3 \\-\varphi_1 & -2\varphi_2 & -3\varphi_3 & -4\varphi_4 \\\varphi_1^2-2\varphi_2 & \varphi_1\varphi_2-3\varphi_3 & \varphi_1\varphi_3-4\varphi_4 & \varphi_1\varphi_4\end{array}\right|,
\end{equation*}
where $\mathscr D$ is the discriminant of \eqref{pol30}. In particular if $P_3=0$ we then find $\dfrac{d}{dx}\left(\lambda\right)=0$. More precisely the differential equation associated to \eqref{pol30} is given by the determinant
\begin{equation}
\label{pol90876}
\left|\begin{array}{ccccc}\frac{dy}{dx} & \varphi_1^\prime &  \varphi_2^\prime &  \varphi_3^\prime & \varphi_4^\prime\\-1 & 4 & 3\varphi_1 & 2\varphi_2 & \varphi_3 \\-y & -\varphi_1 & -2\varphi_2 & -3\varphi_3 & -4\varphi_4  \\-y^2 & \varphi_1^2-2\varphi_2 & \varphi_1\varphi_2-3\varphi_3 & \varphi_1\varphi_3-4\varphi_4 & \varphi_1\varphi_4  \\-y^3 & A & B & C & D \end{array}\right|=0
\end{equation}
with 
\begin{equation}
\label{pol346657577}
A=-\varphi_1^3+3\varphi_1\varphi_2-3\varphi_3;\;B=-\varphi_1^2\varphi_2+2\varphi_2^2+\varphi_1\varphi_3-4\varphi_4;\;C=-\varphi_1^2\varphi_3+2\varphi_2\varphi_3+\varphi_1\varphi_4;\;D=-\varphi_1^2\varphi_4+2\varphi_2\varphi_4.
\end{equation}
The above expressions of the $\alpha_{i,j}$ in terms of the $\varphi_i$ are obtained from their definition \eqref{brioschi} and the Girard-Newton formulas.
\end{example}
\begin{example}
\label{onetotwo}
We analyze another example with $n=2$ that is the case of a separable polynomial
\begin{equation}
\label{pol30LLKK1}
P_\varphi(z)=z^{3}+\varphi_1(x)z^2+\varphi_2(x)z+\varphi_3(x),\;(\varphi_i)_{1\leqslant i\leqslant 3}\in K.
\end{equation}
Then we find that the associated Riccati equation is given by
\begin{equation}
\label{pol908JLHJKLLKJ76}
\left|\begin{array}{cccc}\frac{dy}{dx} & \varphi_1^\prime &  \varphi_2^\prime &  \varphi_3^\prime\\-1 & 3 & 2\varphi_1 & \varphi_2  \\-y & -\varphi_1 & -2\varphi_2 & -3\varphi_3 \\-y^2 & \varphi_1^2-2\varphi_2 & \varphi_1\varphi_2-3\varphi_3 &  \varphi_1\varphi_3  \end{array}\right|=0.
\end{equation}
Let us introduce a restricted type of solutions of a first order differential polynomial $G(y,y^\prime)\in K\left[y,y^\prime\right]\setminus K$ by calling $t$ a solution of $G(y,y^\prime)$ if $t$ lies in a differential field extension $L $ of $K$ with $C_L=C_K$ (same field of constants) and $G(t,t^\prime)=0$, \cite[p. 48]{put2015}. Every of our Riccati equations in this example possesses three distinct algebraic solutions of this restricted type by \cite[Th. 6.2.6]{chambert2005}, since $C_K$ is algebraically closed. Hence by constancy of the cross-ratio of any four given distinct solutions to a Riccati equation, every restricted solution to the Riccati equation \eqref{pol908JLHJKLLKJ76} in the Picard-Vessiot extension over $K$ of the associated linear homogenous second order differential equation, cf. \eqref{khadija}, is algebraic over $K$. Actually it follows from the discussion of \cite[§7.2]{roy2024} that any restricted solution to our Riccati equation lies in the Picard-Vessiot extension over $K$ of its associated linear homogeneous second differentiaI equation. Thus any restricted solution to such a Riccati equation is algebraic over $K$ by constancy of the cross-ratio. Thus by \cite[Th. 2.8]{put2015} the said Riccati equation \eqref{pol908JLHJKLLKJ76} admits a first integral in its function field $K(y)$: a non-trivial constant of the natural derivation associated to \eqref{pol908JLHJKLLKJ76}. 

Let us now suppose that $\varphi_1=0$ then we have from \eqref{pol908JLHJKLLKJ76} the Riccati equation
\begin{equation}
\label{ricdtr}
(27\varphi_2^2+4\varphi_3^3)\dfrac{dy}{dx}+3(2\varphi_3\varphi_2^\prime-3\varphi_2\varphi_3^\prime)y^2-\dfrac{1}{6}\dfrac{d}{dx}(27\varphi_2^2+4\varphi_3^3)y+2(2\varphi_3\varphi_2^\prime-3\varphi_2\varphi_3^\prime)\varphi_3=0.
\end{equation}
Let us further set $\varphi_2=-12$, $\varphi_3=-8x$. Then the Riccati differential equation takes after the suppression of a common factor $64\times9$, the form
\begin{equation}
\label{derivation}
3(x^2-4)\dfrac{dy}{dx}+y^2-xy-8=0.
\end{equation}
Furthermore the associated polynomial is given by
$$P(z)=z^3-12z-8x.$$
Let now $X:=3(x^2-4)\partial_x+(8+xy-y^2)\partial_y$ be the derivation associated to \eqref{derivation}. Let us further introduce
$$Q=xy^3+12y^2+12xy+8(x^2-2)$$
and
$$\Delta=y^2+2xy+4.$$
Then one readily checks that
$$X(Q)=3(2x-y)Q,\;\;\;X(\Delta)=2(2x-y)\Delta.$$
Thus 
$$X\left(\dfrac{Q^2}{\Delta^3}\right)=0$$
and hence $\frac{Q^2}{\Delta^3}$ is a rational first integral of the Riccati equation \eqref{derivation}. This confirms the previous prediction.
\end{example}
\begin{remark}
Given a separable polynomial $P_\varphi(z)$ with coefficients in $K$ of the form \eqref{pol26} we have exhibited a uniquely defined differential equation of the form \eqref{pol98976797} such that every root of such a separable polynomial is a solution to that differential equation. Assume $n\geq3$ and in the notation above that $P_n\not=0$. Then it is interesting to investigate whether our differential equation \eqref{GHKLGHHG} has a Liouvillian transcendental solution in a no new constant differential field extension $L$ of $K$, see \cite[p. 600]{roy2023} for the definition. Then according to \cite[Th. 3.2]{roy2023} our differential equation \eqref{GHKLGHHG} has at least $n$ solutions algebraic of $K$ and its number of algebraic solutions is finite, in case it admits a transcendental Liouvillian solution (solution in the restricted sense: in a differential field extension of with the same field of constants). So potentially our differential equation \eqref{GHKLGHHG} may admit such a Liouvillian transcendental solution since the number of its algebraic over $K$ solutions is at least $(n+1)$. More generally \cite[Th. 3.3]{roy2023} gives precisely when such a transcendental Liouvillian solution should exist. Moreover they express when the number of algebraic over $K$, solutions to a differential equation of the form
$$y^\prime=P_ny^n+P_{n-1}y^{n-1}+\ldots+P_1y+P_0,\quad (P_i)_{0\leqslant i\leqslant n}\in K, \;P_n\not=0$$
is precisely $n$, and the differential equation has a transcendental Liouvillian solution, that differential equation in terms of its $n$ algebraic solutions. Our constructions are different and do not intersect since in case of intersection our differential equation \eqref{GHKLGHHG} would have to be a degree $(n+1)$ polynomial in $y$ on its right hand side, which it is not. Finally in the construction of \cite[prop. 3.6]{roy2023} the leading coefficient $P_n$ must be non zero. 
\end{remark}

\end{document}